\documentclass[a4paper,11pt, reqno]{amsart}
\usepackage{latexsym}
\usepackage{amssymb,amsfonts,amsmath,mathrsfs}
\newtheorem{theorem}{Theorem}[section]

\newtheorem{lemma}{Lemma}[section]
\newtheorem{coro}{Corollary}[section]

\begin{document} 
\title [Non-vanishing of Dirichlet $L$-functions at the central point]{Non-vanishing of Dirichlet $L$-functions at the central point} 

\author{H. M. Bui}
\address{Mathematical Institute, University of Oxford, Oxford, OX1 3LB UK}
\email{hung.bui@maths.ox.ac.uk}
\thanks{The author is supported by an EPSRC Postdoctoral Fellowship.}


\begin{abstract} 
Let $\chi$ be a primitive Dirichlet character modulo $q$ and $L(s,\chi)$ be the Dirichlet $L$-function associated to $\chi$. Using a new two-piece mollifier we show that $L(\tfrac{1}{2},\chi)\ne0$ for at least $34\%$ of the characters in the family. 
\end{abstract} 

\maketitle 

\section{Introduction}

The behavior of $L$-functions and their derivatives inside the critical strip is a very important topic in number theory. In this paper, we study the order of vanishing of Dirichlet $L$-functions at the central point, $s=\tfrac{1}{2}$. Let $\chi$ be a primitive character modulo $q$, and denote by $L(s,\chi)$ the associated $L$-function. It is widely believed that $L(\tfrac{1}{2},\chi)\neq 0$ for all $\chi$. For quadratic characters $\chi$, this appears to have been first conjectured by Chowla [\textbf{\ref{Ch}}].

Balasubramanian and Murty [\textbf{\ref{BM}}] were the first to prove that a positive, though very small, proportion of Dirichlet $L$-functions in the family of primitive characters, to a sufficiently large prime modulus $q$, do not vanish at $s=\tfrac{1}{2}$. This was subsequently improved by Iwaniec and Sarnak [\textbf{\ref{IS}}], who showed that at least $1/3$ of the central values $L(\tfrac{1}{2},\chi)$ are non-zero for all sufficiently large $q$. In this article, we give a modest improvement to the latter result. Our main theorem is

\begin{theorem}
We have
\begin{equation*}
{\sum_{\substack{\chi \ \!\!(\!\!\!\!\!\!\mod\ \!\!q)\\ L({\scriptstyle{\frac{1}{2}}},\chi)\ne0}}{\!\!\!\!\!\!}}^{*} \ 1 \ \geq \big(0.3411+o(1)\big) {\sum_{\chi \ \!\!(\!\!\!\!\!\!\mod\ \!\!q)}{\!\!\!\!\!\!}}^{*} \ 1,
\end{equation*}
where $\sum^{*}$ denotes that the summation is restricted to primitive characters $\chi$.
\end{theorem}
 
Our method is to use a new two-piece mollifier, i.e. the sum of two mollifiers of different shapes. As discussed in [\textbf{\ref{BCY}}], it requires a significant amount of work in studying this and other problems involving two-piece mollifiers, especially when the second mollifier is much more complicated than the usual one (see the definitions in the next section). However, in foreseeing how much improvement can be obtained, one can use some heuristic arguments from the ratios conjecture [\textbf{\ref{CFZ}}] to express various mollified moments of $L$-functions as certain multiple contour integrals. For a variety of examples of such calculations, see [\textbf{\ref{CS}}].
 
There has been a great deal of attention and extensive literature on the non-vanishing of various families of $L$-functions at the center of the critical strip. Most notably, in connection to this work we mention a result of Soundararajan [\textbf{\ref{S2}}], who showed that $L(\tfrac{1}{2},\chi)\neq 0$ for at least $7/8$ of the quadratic characters modulo $q$. For results concerning the central values of high derivatives of $L$-functions, see [\textbf{\ref{KMV}},\textbf{\ref{MV}},\textbf{\ref{BM1}}].  

\subsection{Notation}

For simplicity we only consider even Dirichlet characters $\chi$(mod $q$), i.e. such that $\chi(-1)=1$. The case of odd characters, $\chi(-1)=-1$, is similar. We let $\mathscr{C}_q$ denote the set of primitive characters (mod $q$) and let $\mathscr{C}_{q}^{+}$ denote the subset of characters in $\mathscr{C}_q$ which are even. We put $\varphi^{+}(q) = \frac{1}{2}\varphi^{*}(q)$ where
$$  \varphi^{*}(q) = \sum_{q=dr} \mu(d)\varphi(r)= \big| \mathscr{C}_q\big|.$$
It is not difficult to show that $\big|\mathscr{C}_{q}^{+}\big|= \varphi^{+}(q)+O(1)$. In addition, we write $\sum_{\chi \ \!\!(\textrm{mod}\ \!q)}^{+}$ to indicate that the summation is restricted to $\chi \in \mathscr{C}_{q}^{+}$ and we write $\sum_{a(\textrm{mod}\ q)}^{*}$ to indicate that the summation is restricted to the residues $a(\textrm{mod}\ q)$ which are coprime to $q$. Throughout the paper, we denote $L=\log q$, $y=q^\vartheta$, $y_1=q^{\vartheta_1}$, $y_2=q^{\vartheta_2}$, $P[m]=P(\tfrac{\log y_1/m}{\log y_1})$ and $Q[m]=Q(\tfrac{\log y_2/m}{\log y_2})$, where $P(x)$ and $Q(x)$ are two polynomials satisfying $P(0)=Q(0)=0$.

\section{Mollifier method}

\subsection{Various mollifiers}

We first illustrate the idea of Iwaniec and Sarnak [\textbf{\ref{IS}}]. To each character $\chi\in\mathscr{C}_{q}^{+}$ we associate the function 
\begin{equation}\label{26}
M(\chi)=\sum_{m\leq y}\frac{x_m\chi(m)}{\sqrt{m}},
\end{equation}
where $X=(x_m)$ is a sequence of real numbers supported on $1\leq m\leq y$ with $x_1=1$ and $x_m\ll1$. The purpose of the function $M(\chi)$ is to smooth out or ``mollify'' the large values of $L(\tfrac{1}{2},\chi)$ as we average over $\chi \in \mathscr{C}_{q}^{+}$. Consider
\begin{eqnarray*}
S_1(M)={\sum_{\chi(\text{mod }q)}\!\!\!\!}^+ \ L(\tfrac{1}{2},\chi)M(\chi)\quad\textrm{and}\quad S_2(M)={\sum_{\chi(\text{mod }q)}\!\!\!\!}^+ \ \big|L(\tfrac{1}{2},\chi)M(\chi)\big|^2.
\end{eqnarray*}
The Cauchy inequality implies that
\begin{equation}\label{25}
{\sum_{\substack{\chi(\text{mod }q) \\ L({\scriptstyle{\frac{1}{2}}},\chi)\neq 0}}\!\!\!\!\!}^+  \ 1 \ \geq \ \frac{\big|S_1(M)\big|^2}{S_2(M)}.
\end{equation}
Minimizing the ratio $|S_1(M)|^2/S_2(M)$ with respect to the vector $X=(x_m)$, the optimal mollifier turns out to be
\begin{equation*}
M(\chi)=\sum_{m\leq y}\frac{\mu(m)\chi(m)}{\sqrt{m}}\bigg(1-\frac{\log m}{\log y}\bigg)
\end{equation*}
with $0<\vartheta<\tfrac{1}{2}$. The optimal proportion obtained in \eqref{25} is $\tfrac{1}{3}$, which corresponds to the choice $\vartheta=\tfrac{1}{2}$.

There are two different approaches to improve the results in this and other problems involving mollifiers. One can either extend the length of the Dirichlet polynomial, $y=q^{\vartheta}$, or use some ``better'' mollifiers. The former is certainly much more difficult. For the latter, various mollifiers have been recently proposed by combining mollifiers of different shapes. We give a brief description of these mollifiers and explanation of their failure in applying to our problem below.

Using Soundararajan's idea in [\textbf{\ref{S1}}], Michel and VanderKam [\textbf{\ref{MV}}] studied the ``twisted'' mollifier
\begin{equation*}
M(\chi)=\sum_{m\leq y_1}\frac{\mu(m)\chi(m)P[m]}{\sqrt{m}}+\overline{\varepsilon}_{\chi}\sum_{m\leq y_2}\frac{\mu(m)\overline{\chi}(m)Q[m]}{\sqrt{m}}
\end{equation*}
with $\varepsilon_\chi$ being the root factor in the functional equation $\Lambda(s,\chi)=\varepsilon_\chi\Lambda(1-s,\overline{\chi})$, where
\begin{equation*}
\Lambda(s,\chi)=\bigg(\frac{q}{\pi}\bigg)^{s/2}\Gamma\bigg(\frac{s}{2}\bigg)L(s,\chi).
\end{equation*}
However, the intricate analysis of the arising off-diagonal terms restricts the lengths to $\vartheta_1+\vartheta_2<\tfrac{1}{2}$. Under this limited condition, the optimal proportion obtained is again $\tfrac{1}{3}$. See the discussion at the end of [\textbf{\ref{MV}}].

In [\textbf{\ref{F}}], Feng introduced a mollifier for $\zeta(s)+\tfrac{\zeta'(s)}{\log T}$ and applied it to various problems on the distribution of the zeros of the Riemann zeta-function. In our case, that has the shape
\begin{equation*}
M(\chi)=\sum_{m\leq y_1}\frac{\mu(m)\chi(m)P[m]}{\sqrt{m}}+\frac{1}{L^2}\sum_{m\leq y_2}\frac{(\Lambda*\Lambda*\mu)(m)\chi(m)Q[m]}{\sqrt{m}}.
\end{equation*}
We note that this is already included in the general mollifier \eqref{26} of Iwaniec and Sarnak.

Finally we mention a mollifier type of Lou [\textbf{\ref{L}},\textbf{\ref{BCY}}]. Transforming in the context of Dirichlet $L$-functions, this is
\begin{equation*}
M(\chi)=\sum_{m\leq y_1}\frac{\mu(m)\chi(m)P[m]}{\sqrt{m}}+\varepsilon_\chi\sum_{mn\leq y_2}\frac{\mu_2(m)\chi(m)\overline{\chi}(n)Q[mn]}{\sqrt{mn}}.
\end{equation*}
Carrying out the evaluation of $S_2(M)$ in this case, the only factor involving $P(x)$ in the asymptotic formula for the cross term is $P''(x)$. As a result, when $P(x)=x$, which is optimal in Iwaniec and Sarnak [\textbf{\ref{IS}}], the cross term of $S_2(M)$ vanishes. In fact, it turns out that the optimal choice is $P(x)=x$ and $Q(x)=0$, and hence we do not gain any improvement. 

\subsection{A new mollifier}

Instead we study a two-piece mollifier of the form $\psi(\chi)=\psi_1(\chi)+\psi_2(\chi)$, where
\begin{equation*}
\psi_1(\chi)=\sum_{m\leq y_1}\frac{\mu(m)\chi(m)P[m]}{\sqrt{m}}
\end{equation*}
and
\begin{equation}\label{28}
\psi_2(\chi)=\frac{1}{L}\sum_{mn\leq y_2}\frac{(\log*\mu)(m)\mu(n)\overline{\chi}(m)\chi(n)Q[mn]}{\sqrt{mn}}.
\end{equation}
In the corresponding context of the Riemann zeta-function, this is a more natural mollifier for $\zeta'(s)$, and has been effectively used to show that more than $70.83\%$ of the nontrivial zeros of $\zeta(s)$ are simple on assuming the Riemann Hypothesis and the Generalized Lindel\"of Hypothesis [\textbf{\ref{B}}].

To see that this is also a useful choice of a mollifier for $L(\tfrac{1}{2},\chi)$, we consider the question of constructing a mollifier of different shape rather than the familiar type \eqref{26}. In view of the discussion in the previous subsection, we want our Dirichlet polynomial to contain the character $\overline{\chi}$ and not to involve the root factor $\varepsilon_\chi$. Informally we have
\begin{equation}\label{27}
\frac{1}{L(\tfrac{1}{2},\chi)}=\frac{L(\tfrac{1}{2},\overline{\chi})}{L(\tfrac{1}{2},\overline{\chi})L(\tfrac{1}{2},\chi)}=\sum_{m_1,m_2,n\geq1}\frac{\mu(m_2)\mu(n)\overline{\chi}(m_1m_2)\chi(n)}{\sqrt{m_1m_2n}}Q(m_1,m_2,n),
\end{equation}
for some certain smooth function $Q(u,v,w)$. We note that \eqref{28} is a special case of \eqref{27} by choosing
\begin{equation*}
Q(u,v,w)=L^{-1}(\log u)Q[uvw].
\end{equation*}
Since \eqref{27} also covers the familiar mollifier \eqref{26}, it is probable that studying the general mollifier \eqref{27} would lead to better improvement.

Our Theorem 1.1 is a consequence of the following two theorems.

\begin{theorem}
Suppose $\vartheta_1,\vartheta_2<1$. We have
\begin{eqnarray*}
S_1(\psi)=\varphi^{+}(q)\bigg(P(1)+\tfrac{\vartheta_2}{2}Q_1(1)\bigg)+O(qL^{-1+\varepsilon}),
\end{eqnarray*}
where
\begin{equation*}
Q_1(x)=\int_{0}^{x}Q(u)du.
\end{equation*}
\end{theorem}

\begin{theorem}
Suppose $\vartheta_2<\vartheta_1<\tfrac{1}{2}$. We have
\begin{equation*}
S_2(\psi)=\lambda\varphi^{+}(q)+O(qL^{-1+\varepsilon}),
\end{equation*}
where
\begin{eqnarray*}
\lambda&=&P(1)^2+\tfrac{1}{\vartheta_{1}}\int_{0}^{1}P'(x)^2dx-\vartheta_2 P(1)Q_1(1)+2\vartheta_2\int_{0}^{1}P\big(1-\tfrac{\vartheta_2(1-x)}{\vartheta_1}\big)Q(x)dx\\
&&\quad+\tfrac{\vartheta_2}{\vartheta_1}\int_{0}^{1}P'\big(1-\tfrac{\vartheta_2(1-x)}{\vartheta_1}\big)Q(x)dx+\vartheta_{2}^{2}\int_{0}^{1}(1-x)Q(x)^2dx+\tfrac{\vartheta_2}{2}\int_{0}^{1}(1-x)^2Q'(x)^2dx\\
&&\quad\quad-\tfrac{\vartheta_{2}^{2}}{4}Q_1(1)^2+\tfrac{\vartheta_2}{4}\int_{0}^{1}Q(x)^2dx.
\end{eqnarray*}
\end{theorem}

\noindent\textit{Deduction of Theorem 1.1}. We take $\vartheta_1=\vartheta_2=\tfrac{1}{2}$, $Q(x)=0.9x$ and
\begin{displaymath}
P(x)=1.05x-0.05x^2.
\end{displaymath}
These above theorems and \eqref{25} lead to Theorem 1.1.

\subsection{Setting up}

We collect the results in [\textbf{\ref{IS}}] to obtain
\begin{eqnarray*}
\frac{1}{\varphi^{+}(q)}{\sum_{\chi(\textrm{mod }q)}\!\!\!\!}^+ \ L(\tfrac{1}{2},\chi)\psi_1(\chi)=P(1)+O(L^{-1+\varepsilon})
\end{eqnarray*}
for $0<\vartheta_1<1$, and
\begin{eqnarray*}
\frac{1}{\varphi^{+}(q)}{\sum_{\chi(\textrm{mod }q)}\!\!\!\!}^+ \ \big|L(\tfrac{1}{2},\chi)\psi_1(\chi)\big|^2=P(1)^2+\tfrac{1}{\vartheta_1}\int_{0}^{1}P'(x)^2dx+O(L^{-1+\varepsilon})
\end{eqnarray*}
for $0<\vartheta_1<\tfrac{1}{2}$. Hence we are left to consider
\begin{eqnarray*}
I={\sum_{\chi(\text{mod }q)}\!\!\!\!}^+ \ L(\tfrac{1}{2},\chi)\psi_2(\chi),\quad J_1={\sum_{\chi(\text{mod }q)}\!\!\!\!}^+ \ \big|L(\tfrac{1}{2},\chi)|^2\psi_1(\chi)\psi_2(\overline{\chi})
\end{eqnarray*}
and
\begin{eqnarray*}
J_2={\sum_{\chi(\text{mod }q)}\!\!\!\!}^+ \ \big|L(\tfrac{1}{2},\chi)\psi_2(\chi)\big|^2.
\end{eqnarray*}

\subsection{Introducing the shifts}

It is no more difficult to work with shifted moments. Instead we will consider the following more general sums

\begin{eqnarray}
&&I(\alpha)={\sum_{\chi(\text{mod }q)}\!\!\!\!}^+ \ L(\tfrac{1}{2}+\alpha,\chi)\psi_2(\chi),\label{5}\\
&&J_1(\alpha,\beta)={\sum_{\chi(\text{mod }q)}\!\!\!\!}^+ \ L(\tfrac{1}{2}+\alpha,\chi)L(\tfrac{1}{2}+\beta,\overline{\chi})\psi_1(\chi)\psi_2(\overline{\chi})\label{6}
\end{eqnarray}
and
\begin{equation}\label{7}
J_2(\alpha,\beta)={\sum_{\chi(\text{mod }q)}\!\!\!\!}^+ \ L(\tfrac{1}{2}+\alpha,\chi)L(\tfrac{1}{2}+\beta,\overline{\chi})\big|\psi_2(\chi)\big|^2.
\end{equation}

Our main goal in the rest of the paper is to prove the following lemmas.

\begin{lemma}
Suppose $\vartheta_2<1$. Uniformly for $\alpha\ll L^{-1}$ we have
\begin{eqnarray*}
I(\alpha)=\varphi^{+}(q)\bigg(\vartheta_2\int_{0}^{1}y^{-\alpha(1-x)}Q(x)dx-\tfrac{\vartheta_2}{2}Q_1(1)\bigg)+O(qL^{-1+\varepsilon}).
\end{eqnarray*}
\end{lemma}

\begin{lemma}
Suppose $\vartheta_2<\vartheta_1<\tfrac{1}{2}$. Uniformly for $\alpha,\beta\ll L^{-1}$ we have
\begin{eqnarray}\label{19}
&&\!\!\!\!\!\!\!\!\!\!J_1(\alpha,\beta)=\varphi^{+}(q)\frac{d^2}{dadb}\bigg\{\tfrac{\vartheta_2}{\vartheta_1}\int_{0}^{1}\int_{0}^{1}\int_{0}^{x}y_{1}^{\beta a}y_{2}^{\alpha b-\beta u}(qy_{1}^{a}y_{2}^{b-u})^{-(\alpha+\beta)t}\big(1+\vartheta_1a+\vartheta_2(b-u)\big)\\
&&\!\!\!\!\!\!\!\qquad P\big(1-\tfrac{\vartheta_2(1-x)}{\vartheta_1}+a\big)Q(x-u+b)dudxdt-\tfrac{\vartheta_2}{2\vartheta_1}\int_{0}^{1}\int_{0}^{1}y_{1}^{\beta a}y_{2}^{\alpha b}(qy_{1}^{a}y_{2}^{b})^{-(\alpha+\beta)t}\big(1+\vartheta_1a+\vartheta_2b\big)\nonumber\\
&&\!\!\!\!\!\!\!\!\!\!\qquad\qquad\qquad P\big(1-\tfrac{\vartheta_2(1-x)}{\vartheta_1}+a\big)Q_1(x+b)dxdt\bigg\}\bigg|_{a=b=0}+O(qL^{-1+\varepsilon}).\nonumber
\end{eqnarray}
\end{lemma}

\begin{lemma}
Suppose $\vartheta_2<\tfrac{1}{2}$. Uniformly for $\alpha,\beta\ll L^{-1}$ we have
\begin{eqnarray*}
J_{2}(\alpha,\beta)&=&\varphi^{+}(q)\frac{d^2}{dadb}\bigg\{\tfrac{\vartheta_2}{2}\int_{0}^{1}y_{2}^{\alpha b+\beta a}(qy_{2}^{a+b})^{-(\alpha+\beta)t}\big(1+\vartheta_2(a+b)\big)(1-x)^2Q(x+a)Q(x+b)dx\\
&&\!\!\!\!\!\!\!\!\!\!\!\!\!\!\!\!\!\!\!\!\!\!\!\!\!+\vartheta_2\int_{0}^{1}\int_{0}^{x}\int_{0}^{x}y_{2}^{\alpha b+\beta a-\alpha u-\beta v}(qy_{2}^{a+b-u-v})^{-(\alpha+\beta)t}\\
&&\qquad\qquad\qquad\big(1+\vartheta_2(a+b-u-v)\big)Q(x-u+a)Q(x-v+b)dudvdx\\
&&\!\!\!\!\!\!\!\!\!\!\!\!\!\!\!\!\!\!\!\!\!\!\!\!\!-\tfrac{\vartheta_2}{2}\int_{0}^{1}\int_{0}^{x}y_{2}^{\alpha b+\beta a-\alpha u}(qy_{2}^{a+b-u})^{-(\alpha+\beta)t}\big(1+\vartheta_2(a+b-u)\big)Q(x-u+a)Q_1(x+b)dudx\\
&&\!\!\!\!\!\!\!\!\!\!\!\!\!\!\!\!\!\!\!\!\!\!\!\!\!-\tfrac{\vartheta_2}{2}\int_{0}^{1}\int_{0}^{x}y_{2}^{\alpha a+\beta b-\beta u}(qy_{2}^{a+b-u})^{-(\alpha+\beta)t}\big(1+\vartheta_2(a+b-u)\big)Q(x-u+a)Q_1(x+b)dvdx\\
&&\!\!\!\!\!\!\!\!\!\!\!\!\!\!\!\!\!\!\!\!\!\!\!\!\!+\tfrac{\vartheta_2}{4}\int_{0}^{1}y_{2}^{\alpha b+\beta a}(qy_{2}^{a+b})^{-(\alpha+\beta)t}\big(1+\vartheta_2(a+b)\big)Q_1(x+a)Q_1(x+b)dx\bigg\}\bigg|_{a=b=0}+O(qL^{-1+\varepsilon}).
\end{eqnarray*}
\end{lemma}

The following corollaries are direct consequences.

\begin{coro}
Suppose $\vartheta_2<1$. We have
\begin{eqnarray*}
\frac{1}{\varphi^{+}(q)}{\sum_{\chi(\emph{mod }q)}\!\!\!\!}^+ \ L(\tfrac{1}{2},\chi)\psi_2(\chi)=\tfrac{\vartheta_2}{2}Q_1(1)+O(L^{-1+\varepsilon}).
\end{eqnarray*}
\end{coro}

\begin{coro}
Suppose $\vartheta_2<\vartheta_1<\tfrac{1}{2}$. We have
\begin{eqnarray*}
\frac{1}{\varphi^{+}(q)}{\sum_{\chi(\emph{mod }q)}\!\!\!\!}^+ \ \big|L(\tfrac{1}{2},\chi)\big|^2\psi_1(\chi)\psi_2(\overline{\chi})&=&-\tfrac{\vartheta_2}{2} P(1)Q_1(1)+\vartheta_2\int_{0}^{1}P\big(1-\tfrac{\vartheta_2(1-x)}{\vartheta_1}\big)Q(x)dx\\
&&\qquad+\tfrac{\vartheta_2}{2\vartheta_1}\int_{0}^{1}P'\big(1-\tfrac{\vartheta_2(1-x)}{\vartheta_1}\big)Q(x)dx+O(L^{-1+\varepsilon}).
\end{eqnarray*}
\end{coro}

\begin{coro}
Suppose $\vartheta<\tfrac{1}{2}$. We have
\begin{eqnarray*}
\frac{1}{\varphi^{+}(q)}{\sum_{\chi(\emph{mod }q)}\!\!\!\!}^+ \ \big|L(\tfrac{1}{2},\chi)\psi_2(\chi)\big|^2&=&\vartheta_{2}^{2}\int_{0}^{1}(1-x)Q(x)^2dx+\tfrac{\vartheta_2}{2}\int_{0}^{1}(1-x)^2Q'(x)^2dx\\
&&\quad\quad-\tfrac{\vartheta_{2}^{2}}{4}Q_1(1)^2+\tfrac{\vartheta_2}{4}\int_{0}^{1}Q(x)^2dx+O(L^{-1+\varepsilon}).
\end{eqnarray*}
\end{coro}

Collecting these results, we obtain Theorem 2.1 and Theorem 2.2.

\section{Various lemmas}

In this section we collect some preliminary results which we will use later. 

\begin{lemma}
For $(mn,q)=1$ we have
\begin{displaymath}
{\sum_{\chi(\emph{mod}\ q)}{\!\!\!\!\!\!}}^{+}\ \chi(m)\overline{\chi}(n)=\tfrac{1}{2}\sum_{\substack{q=dr\\r|m\pm n}}\mu(d)\varphi(r).
\end{displaymath}
\end{lemma}
\begin{proof}
This is standard. See, for example, Lemma 4.1 [\textbf{\ref{BM1}}].
\end{proof}

\begin{lemma}
Let
\begin{equation}\label{9}
V(x)=\frac{1}{2\pi i}\int_{(2)}e^{s^2}x^{-s}\frac{ds}{s}.
\end{equation}
Then for $\chi\in\mathscr{C}_{q}^{+}$ and any $B>0$ we have 
\begin{eqnarray*}
L({\scriptstyle{\frac{1}{2}}}+\alpha,\chi)=\sum_{m\geq1}\frac{\chi(m)}{m^{1/2+\alpha}}V\bigg(\frac{m}{q^{1+\varepsilon}}\bigg)+O(q^{-B}).
\end{eqnarray*}
\end{lemma}
\begin{proof}
Consider
\begin{equation*}
A=\frac{1}{2\pi i}\int_{(2)}X^se^{s^2}L(\tfrac{1}{2}+\alpha+s,\chi)\frac{ds}{s}.
\end{equation*}
We move the line of integration to $\textrm{Re}(s)=-N$, crossing a simple pole at $s=0$. On the new contour, we use the decay of $e^{s^2}$ and the bound $L(\sigma+it,\chi)\ll \big(q(1+|t|)\big)^{1/2-\sigma}$ for $\sigma<0$. In doing so we obtain 
\begin{equation*}
A=L(\tfrac{1}{2}+\alpha,\chi)+O(X^{-N}q^N).
\end{equation*}
We now take $X=q^{1+\varepsilon}$ and choose $N$ large enough with respect to $\varepsilon$. Finally expressing the $L$-function in the integral as Dirichlet series we obtain the lemma.
\end{proof}

\begin{lemma}
Let
\begin{displaymath}
\mathscr{A}(h,k)={\sum_{\chi(\emph{mod}\ \!q)}{\!\!\!\!\!}}^{+}\ L({\scriptstyle{\frac{1}{2}}}+\alpha,\chi)\overline{\chi}(h)\chi(k).
\end{displaymath}
Then for $(hk,q)=1$ and uniformly for $\alpha\ll L^{-1}$ we have
\begin{eqnarray}\label{3}
\mathscr{A}(h,k)=\varphi^{+}(q){\sum_{mk=h}\!\!}^{*}\ \ \frac{1}{m^{1/2+\alpha}}V\bigg(\frac{m}{q^{1+\varepsilon}}\bigg)+O(E_1(h,k)+q^{-B}),
\end{eqnarray}
where $E_1(h,k)$ satisfies
\begin{displaymath}
\sum_{hk\leq y}\frac{E_1(h,k)}{\sqrt{hk}}\ll (yq)^{1/2+\varepsilon}.
\end{displaymath}
\end{lemma}
\begin{proof}
In view of Lemma 3.1 and Lemma 3.2 we have
\begin{eqnarray}\label{4}
\mathscr{A}(h,k)=\tfrac{1}{2}\sum_{q=dr}\mu(d)\varphi(r){\sum_{r|mk\pm h}{\!\!\!\!\!}}^{*}\ \ \frac{1}{m^{1/2+\alpha}}V\bigg(\frac{m}{q^{1+\varepsilon}}\bigg)+O(q^{-B}).
\end{eqnarray}
The diagonal terms $mk=h$ gives the main term visible in \eqref{3}. All the other terms in \eqref{4} contribute at most
\begin{displaymath}
E_1(h,k)=\sum_{mk\ne h}\frac{(mk\pm h,q)}{\sqrt{m}}\bigg|V\bigg(\frac{m}{q^{1+\varepsilon}}\bigg)\bigg|.
\end{displaymath}
Using the estimate $V(x)\ll(1+|x|)^{-1}$ one can easily show that
\begin{equation*}
\sum_{hk\leq y}\frac{E_1(h,k)}{\sqrt{hk}}\ll (yq)^{1/2+\varepsilon}.
\end{equation*}
The lemma follows.
\end{proof}

\begin{lemma}
Let $G(s)=e^{s^2}p(s)$ and $p(s)=\tfrac{(\alpha+\beta)^2-4s^2}{(\alpha+\beta)^2}$. Let
\begin{equation}\label{12}
W_{\alpha,\beta}^{\pm}(x)=\frac{1}{2\pi i}\int_{(2)}G(s)g_{\alpha,\beta}^{\pm}(s)x^{-s}\frac{ds}{s},
\end{equation}
where
\begin{displaymath}
g_{\alpha,\beta}^{+}(s)=\frac{\Gamma(\frac{1/2+\alpha+s}{2})\Gamma(\frac{1/2+\beta+s}{2})}{\Gamma(\frac{1/2+\alpha}{2})\Gamma(\frac{1/2+\beta}{2})}\quad \textrm{and}\quad \ g_{\alpha,\beta}^{-}(s)=\frac{\Gamma(\frac{1/2-\alpha+s}{2})\Gamma(\frac{1/2-\beta+s}{2})}{\Gamma(\frac{1/2+\alpha}{2})\Gamma(\frac{1/2+\beta}{2})}.
\end{displaymath}
Then for $\chi\in\mathscr{C}_{q}^{+}$ we have 
\begin{eqnarray*}
L({\scriptstyle{\frac{1}{2}}}+\alpha,\chi)L({\scriptstyle{\frac{1}{2}}}+\beta,\overline{\chi})&=&\sum_{m,n\geq1}\frac{\chi(m)\overline{\chi}(n)}{m^{1/2+\alpha}n^{1/2+\beta}}W_{\alpha,\beta}^{+}\bigg(\frac{\pi mn}{q}\bigg)\nonumber\\
&&\quad+\bigg(\frac{q}{\pi}\bigg)^{-\alpha-\beta}\sum_{m,n\geq1}\frac{\overline{\chi}(m)\chi(n)}{m^{1/2-\alpha}n^{1/2-\beta}}W_{\alpha,\beta}^{-}\bigg(\frac{\pi mn}{q}\bigg).
\end{eqnarray*}
\end{lemma}
\begin{proof}
This is Lemma 4.2 in [\textbf{\ref{BM1}}].
\end{proof}

\begin{lemma}
Let
\begin{displaymath}
\mathscr{B}(h,k)={\sum_{\chi(\emph{mod}\ \!q)}{\!\!\!\!\!}}^{+}\ L({\scriptstyle{\frac{1}{2}}}+\alpha,\chi)L({\scriptstyle{\frac{1}{2}}}+\beta,\overline{\chi})\overline{\chi}(h)\chi(k).
\end{displaymath}
Then for $(hk,q)=1$ and uniformly for $\alpha,\beta\ll L^{-1}$ we have
\begin{eqnarray}\label{2}
\mathscr{B}(h,k)&=&\varphi^{+}(q)\bigg({\sum_{mk=nh}{\!\!\!\!}}^{*}\ \ \frac{W_{\alpha,\beta}^{+}(\frac{\pi mn}{q})}{m^{1/2+\alpha}n^{1/2+\beta}}+\bigg(\frac{q}{\pi}\bigg)^{-\alpha-\beta}{\sum_{mh=nk}{\!\!\!\!}}^{*}\ \ \frac{W_{\alpha,\beta}^{-}(\frac{\pi mn}{q})}{m^{1/2-\alpha}n^{1/2-\beta}}\bigg)\nonumber\\
&&\qquad\qquad+O(E_2(h,k)),
\end{eqnarray}
where $E_2(h,k)$ satisfies
\begin{displaymath}
\sum_{hk\leq y}\frac{E_2(h,k)}{\sqrt{hk}}\ll (yq)^{1/2+\varepsilon}.
\end{displaymath}
\end{lemma}
\begin{proof}
In view of Lemma 3.1 and Lemma 3.4 we have
\begin{eqnarray}\label{1}
\mathscr{B}(h,k)&=&\tfrac{1}{2}\sum_{q=dr}\mu(d)\varphi(r){\sum_{r|mk\pm nh}{\!\!\!\!\!\!}}^{*}\ \ \frac{W_{\alpha,\beta}^{+}(\frac{\pi mn}{q})}{m^{1/2+\alpha}n^{1/2+\beta}}\nonumber\\
&&\quad+\tfrac{1}{2}\bigg(\frac{q}{\pi}\bigg)^{-\alpha-\beta}{\sum_{q=dr}\mu(d)\varphi(r)\sum_{r|mh\pm nk}{\!\!\!\!\!\!}}^{*}\ \ \frac{W_{\alpha,\beta}^{-}(\frac{\pi mn}{q})}{m^{1/2-\alpha}n^{1/2-\beta}}.
\end{eqnarray}
The diagonal terms $mk=nh$ and $mh=nk$ in the first and second sums on the right-hand side of \eqref{1}, respectively, give the main term visible in \eqref{2}. All the other terms in \eqref{1} contribute at most
\begin{displaymath}
E_2(h,k)=\sum_{mk\ne nh}\frac{(mk\pm nh,q)}{\sqrt{mn}}\bigg|W_{\alpha,\beta}^{\pm}\bigg(\frac{\pi mn}{q}\bigg)\bigg|.
\end{displaymath}
Using the estimate $W_{\alpha,\beta}^{\pm}(x)\ll(1+|x|)^{-1}$ one can easily show that
\begin{equation*}
\sum_{hk\leq y}\frac{E_2(h,k)}{\sqrt{hk}}\ll (yq)^{1/2+\varepsilon}.
\end{equation*}
This completes the proof of the lemma.
\end{proof}

\begin{lemma}
Suppose $y_2\leq y_1$, $|z|\ll(\log y_1)^{-1}$, and that $F_1$ and $F_2$ are smooth functions. Then
\begin{eqnarray*}
&&\sum_{n\leq y_2}\frac{d_k(n)}{n}\bigg(\frac{y_2}{n}\bigg)^{z}F_1\bigg(\frac{\log y_1/n}{\log y_1}\bigg)F_2\bigg(\frac{\log y_2/n}{\log y_2}\bigg)\\
&&\qquad=\frac{(\log y_2)^k}{(k-1)!}\int_{0}^{1}y_{2}^{zx}(1-x)^{k-1}F_1\bigg(1-\frac{(1-x)\log y_2}{\log y_1}\bigg)F_2(x)dx+O((\log y_2)^{k-1}).
\end{eqnarray*}
\end{lemma}
\begin{proof}
This is Lemma 4.4 in [\textbf{\ref{BCY}}].
\end{proof}

\begin{lemma}
Suppose $-1\leq\sigma\leq0$. We have
\begin{equation*}
\sum_{n\leq y}\frac{d_k(n)}{n}\bigg(\frac{y}{n}\bigg)^\sigma\ll (\log y)^{k-1}\min\big\{|\sigma|^{-1},\log y\big\}.
\end{equation*}
\end{lemma}
\begin{proof}
This is Lemma 4.6 in [\textbf{\ref{BCY}}].
\end{proof}

\subsection{Mellin transform pairs}

Let $P(x)=\sum_{i\geq1}a_ix^i$ and $Q(x)=\sum_{j\geq1}b_jx^j$. We note the Mellin transform pairs
\begin{equation}\label{11}
P[h]=\sum_{i\geq1}\frac{a_ii!}{(\log y_1)^i}\frac{1}{2\pi i}\int_{(2)}\frac{y_{1}^{u}}{u^{i+1}}h^{-u}du
\end{equation}
and
\begin{eqnarray}\label{8}
Q[h]=\sum_{j\geq1}\frac{b_jj!}{(\log y_2)^j}\frac{1}{2\pi i}\int_{(2)}\frac{y_{2}^{u}}{u^{j+1}}h^{-u}du.
\end{eqnarray}

\section{Evaluating $I(\alpha)$}

\subsection{Reduction to a contour integral}

We recall that $I(\alpha)$ is defined by \eqref{5}. Writing out the definition of $M_2(\chi)$ we have
\begin{equation*}
I(\alpha)=\frac{1}{L}\sum_{hk\leq y_2}\frac{(\log*\mu)(h)\mu(k)Q[hk]}{\sqrt{hk}}{\sum_{\chi(\textrm{mod}\ \!q)}{\!\!\!\!\!}}^{+}\ L({\scriptstyle{\frac{1}{2}}}+\alpha,\chi)\overline{\chi}(h)\chi(k).
\end{equation*}
By Lemma 3.3, we can write $I(\alpha)=I'(\alpha)+O((y_2q)^{1/2+\varepsilon})$, where
\begin{equation*}
I'(\alpha)=\frac{\varphi^{+}(q)}{L}{\sum_{\substack{hk\leq y_2\\mk=h}}\!\!}^{*}\ \frac{(\log*\mu)(h)\mu(k)Q[hk]}{\sqrt{hk}m^{1/2+\alpha}}V\bigg(\frac{m}{q^{1+\varepsilon}}\bigg).
\end{equation*}
Using \eqref{8} and \eqref{9} we obtain
\begin{equation*}
I'(\alpha)=\frac{\varphi^{+}(q)}{L}\sum_{j\geq1}\frac{b_jj!}{(\log y_2)^j}\bigg(\frac{1}{2\pi i}\bigg)^2\int_{(2)}\int_{(2)}e^{s^2}q^{(1+\varepsilon)s}y_{2}^{u}{\sum_{mk=h_1h_2}\!\!\!\!\!}^{*}\ \ \frac{(\log h_1)\mu(h_2)\mu(k)}{(h_1h_2k)^{1/2+u}m^{1/2+\alpha+s}}\frac{ds}{s}\frac{du}{u^{j+1}}.
\end{equation*}
The sum in the integrand is
\begin{equation*}
-\frac{d}{d\gamma}{\sum_{mk=h_1h_2}\!\!\!\!\!}^{*}\ \ \frac{\mu(h_2)\mu(k)}{h_{1}^{1/2+u+\gamma}(h_2k)^{1/2+u}m^{1/2+\alpha+s}}\bigg|_{\gamma=0}.
\end{equation*}
Hence 
\begin{equation}\label{23}
I'(\alpha)=-\frac{\varphi^{+}(q)}{L}\sum_{j\geq1}\frac{b_jj!}{(\log y_2)^j}\frac{\partial}{\partial\gamma}I''(\alpha,\gamma)\bigg|_{\gamma=0},
\end{equation}
where
\begin{equation*}
I''(\alpha,\gamma)=\bigg(\frac{1}{2\pi i}\bigg)^2\int_{(2)}\int_{(2)}e^{s^2}q^{(1+\varepsilon)s}y_{2}^{u}{\sum_{mk=h_1h_2}\!\!\!\!\!}^{*}\ \ \frac{\mu(h_2)\mu(k)}{h_{1}^{1/2+u+\gamma}(h_2k)^{1/2+u}m^{1/2+\alpha+s}}\frac{ds}{s}\frac{du}{u^{j+1}}.
\end{equation*}
We note that here and throughout the paper, we take $\gamma,\gamma_1,\gamma_2\in\mathbb{C}$ and $\gamma,\gamma_1,\gamma_2\ll L^{-1}$. Some standard calculations give
\begin{equation}\label{13}
{\sum_{mk=h_1h_2}\!\!\!\!\!}^{*}\ \ \frac{\mu(h_2)\mu(k)}{h_{1}^{1/2+u+\gamma}(h_2k)^{1/2+u}m^{1/2+\alpha+s}}=A(\alpha,\gamma,u,s)\frac{\zeta(1+\alpha+\gamma+u+s)\zeta(1+2u)}{\zeta(1+\gamma+2u)\zeta(1+\alpha+u+s)},
\end{equation}
where $A(\alpha,\gamma,u,s)$ is an arithmetical factor given by some Euler product that is absolutely and uniformly convergent in some product of fixed half-planes containing the origin. We first move the $u$-contour to $\textrm{Re}(u)=\delta$, and then move the $s$-contour to $\textrm{Re}(s)=-(1-\varepsilon/2)\delta$, where $\delta>0$ is some fixed small constant such that the arithmetical factor converges absolutely. In doing so we only cross a simple pole at $s=0$. By bounding the integral by absolute values, the contribution along the new line is
\begin{equation*}
\ll q^{-(1+\varepsilon)(1-\varepsilon/2)\delta}y_{2}^{\delta}\ll (y_2q^{-1})^\delta.
\end{equation*} 
Thus
\begin{eqnarray*}
I''(\alpha,\gamma)=\frac{1}{2\pi i}\int_{(\delta)}y_{2}^{u}A(\alpha,\gamma,u,0)\frac{\zeta(1+\alpha+\gamma+u)\zeta(1+2u)}{\zeta(1+\gamma+2u)\zeta(1+\alpha+u)}\frac{du}{u^{j+1}}+O(q^{-\varepsilon}).
\end{eqnarray*}
Moving the contour to $\textrm{Re}(u)\asymp L^{-1}$ and bounding the integral trivially show that $I''(\alpha,\gamma)\ll L^{j}$. Hence by the Cauchy theorem we have
\begin{equation}\label{10}
\frac{\partial}{\partial\gamma}I''(\alpha,\gamma)\bigg|_{\gamma=0}=K_{1}(\alpha)+K_{2}(\alpha)+O(L^{j}),
\end{equation}
where
\begin{equation*}
K_{1}(\alpha)=\frac{1}{2\pi i}\int_{(L^{-1})}y_{2}^{u} A(\alpha,0,u,0)\frac{\zeta'(1+\alpha+u)}{\zeta(1+\alpha+u)}\frac{du}{u^{j+1}}
\end{equation*}
and
\begin{equation*}
K_{2}(\alpha)=-\frac{1}{2\pi i}\int_{(L^{-1})}y_{2}^{u} A(\alpha,0,u,0)\frac{\zeta'(1+2u)}{\zeta(1+2u)}\frac{du}{u^{j+1}}.
\end{equation*}

\subsection{Evaluation of $K_{1}(\alpha)$ and $K_{2}(\alpha)$}

\begin{lemma}
With $K_{11}(\alpha)$ and $K_{12}(\alpha)$ defined above we have
\begin{equation}\label{15}
K_{1}(\alpha)=-\frac{A(\alpha,0,0,0)(\log y_2)^{j+1}}{j!}\int_{0}^{1}y_{2}^{-\alpha(1-x)}x^{j}dx+O(L^{j})
\end{equation}
and
\begin{equation*}
K_{2}(\alpha)=\frac{A(\alpha,0,0,0)(\log y_2)^{j+1}}{2(j+1)!}+O(L^{j}).
\end{equation*}
\end{lemma}
\begin{proof}
We shall illustrate the proof for $K_{1}(\alpha)$. The case of $K_{2}(\alpha)$ can be treated similarly.

By bounding the integral with absolute value we have $K_{1}(\alpha)\ll L^{j+1}$. Denote by $K_{1}'(\alpha)$ the same integral as $K_{1}(\alpha)$ but with $A(\alpha,0,u,0)$ being replaced by $A(\alpha,0,0,0)$. Then we have $K_{1}(\alpha)=K_{1}'(\alpha)+O(L^j)$. As in the proof of the prime number theorem, $K_{1}'(\alpha)$ is captured by the residues at $u=0$ and $u=-\alpha$, with an error of size $O((\log y_2)^{-B})$ for arbitrarily large $B$. Also, we can express the sum of the residues as
\begin{equation*}
A(\alpha,0,0,0)\frac{1}{2\pi i}\oint y_{2}^{u}\frac{\zeta'(1+\alpha+u)}{\zeta(1+\alpha+u)}\frac{du}{u^{j+1}},
\end{equation*}
where the contour is a circle with radius $\asymp L^{-1}$ enclosing the origin. We note that this is trivially bounded by $O(L^{j+1})$. Hence taking the first terms in the Taylor series of the zeta-functions we obtain
\begin{equation*}
K_{1}(\alpha)=-A(\alpha,0,0,0)\frac{1}{2\pi i}\oint y_{2}^{u}\frac{du}{(\alpha+u)u^{j+1}}+O(L^{j}).
\end{equation*}
We use the identity
\begin{equation*}
\frac{1}{\alpha+u}=\int_{1/y_2}^{1}t^{\alpha+u-1}dt+\frac{y_{2}^{-\alpha-u}}{\alpha+u},
\end{equation*}
which is valid for all complex numbers $\alpha,u$. The contribution of the second term to the integral is
\begin{equation*}
-A(\alpha,0,0,0)y_{2}^{-\alpha}\frac{1}{2\pi i}\oint \frac{du}{(\alpha+u)u^{j+1}}.
\end{equation*}
This integral can be seen to vanish by taking the contour to be arbitrarily large. Thus
\begin{eqnarray*}
K_{1}(\alpha)&=&-A(\alpha,0,0,0)\int_{1/y_2}^{1}t^{\alpha}\frac{1}{2\pi i}\oint (y_2t)^{u}\frac{du}{u^{j+1}}\frac{dt}{t}+O(L^{j})\\
&=&-A(\alpha,0,0,0)\int_{1/y_2}^{1}t^{\alpha}\frac{(\log y_2t)^{j}}{j!}\frac{dt}{t}+O(L^{j}).
\end{eqnarray*}
We can write it in a compact form as in \eqref{15} after the substitution $t=y_{2}^{-(1-x)}$.
\end{proof}

\subsection{Deduction of Lemma 2.1}

By Lemma 4.1, \eqref{10} and \eqref{23} we have
\begin{eqnarray*}
I'(\alpha)=\varphi^{+}(q)A(\alpha,0,0,0)\bigg(\vartheta_2\int_{0}^{1}y_{2}^{-\alpha(1-x)}Q(x)dx-\tfrac{\vartheta_2}{2}Q_1(1)\bigg)+O(qL^{-1+\varepsilon}).
\end{eqnarray*}

It is now sufficient to verify that $A(0,0,0,0)=1$. Taking $\alpha=\gamma=0$ and $u=s$ in \eqref{13} we have
\begin{equation*}
A(0,0,s,s)={\sum_{mk=h_1h_2}\!\!\!\!\!}^{*}\ \ \frac{\mu(h_2)\mu(k)}{(h_{1}h_2km)^{1/2+s}}=1
\end{equation*}
for all $s$. This completes the proof of Lemma 2.1.

\section{Evaluating $J_1(\alpha,\beta)$}

\subsection{Reduction to a contour integral}

We recall that $J_1(\alpha,\beta)$ is defined by \eqref{6}. Writing out the definitions of $M_1(\chi)$ and $M_2(\chi)$ we have
\begin{equation*}
J_1(\alpha,\beta)=\frac{1}{L}\sum_{\substack{l\leq y_1\\hk\leq y_2}}\frac{\mu(l)P[l](\log*\mu)(h)\mu(k)Q[hk]}{\sqrt{lhk}}{\sum_{\chi(\textrm{mod}\ \!q)}{\!\!\!\!\!}}^{+}\ L({\scriptstyle{\frac{1}{2}}}+\alpha,\chi)L(\tfrac{1}{2}+\beta,\overline{\chi})\chi(lh)\overline{\chi}(k).
\end{equation*}
By Lemma 3.5, we can write $J_1(\alpha,\beta)=J_{1}^{+}(\alpha,\beta)+J_{1}^{-}(\alpha,\beta)+O((y_1y_2q)^{1/2+\varepsilon})$, where
\begin{equation*}
J_{1}^{+}(\alpha,\beta)=\frac{\varphi^{+}(q)}{L}{\sum_{\substack{l\leq y_1\\hk\leq y_2\\mlh=nk}}\!\!\!\!}^{*}\ \ \frac{\mu(l)P[l](\log*\mu)(h)\mu(k)Q[hk]}{\sqrt{lhk}m^{1/2+\alpha}n^{1/2+\beta}}W_{\alpha,\beta}^{+}\bigg(\frac{\pi mn}{q}\bigg)
\end{equation*}
and
\begin{equation*}
J_{1}^{-}(\alpha,\beta)=\frac{\varphi^{+}(q)}{L}\bigg(\frac{q}{\pi}\bigg)^{-\alpha-\beta}{\sum_{\substack{l\leq y_1\\hk\leq y_2\\mk=nlh}}\!\!\!\!}^{*}\ \ \frac{\mu(l)P[l](\log*\mu)(h)\mu(k)Q[hk]}{\sqrt{lhk}m^{1/2-\alpha}n^{1/2-\beta}}W_{\alpha,\beta}^{-}\bigg(\frac{\pi mn}{q}\bigg).
\end{equation*}
The treatments of $J_{1}^{+}(\alpha,\beta)$ and $J_{1}^{-}(\alpha,\beta)$ are similar. In fact, it is easy to show that
\begin{equation*}
J_{1}^{-}(\alpha,\beta)=\big(q^{-\alpha-\beta}+O(L^{-1})\big)J_{1}^{+}(-\beta,-\alpha).
\end{equation*}

Using \eqref{11}, \eqref{8} and \eqref{12} we obtain
\begin{eqnarray*}
J_{1}^{+}(\alpha,\beta)&=&\frac{\varphi^{+}(q)}{L}\sum_{i,j\geq1}\frac{a_ib_ji!j!}{(\log y_1)^i(\log y_2)^j}\bigg(\frac{1}{2\pi i}\bigg)^3\int_{(2)}\int_{(2)}\int_{(2)}G(s)g_{\alpha,\beta}^{+}(s)\bigg(\frac{q}{\pi}\bigg)^{s}y_{1}^{u}y_{2}^{v}\\
&&\qquad\qquad{\sum_{mlh=nk}\!\!\!\!}^{*}\ \ \frac{\mu(l)(\log*\mu)(h)\mu(k)}{l^{1/2+u}(hk)^{1/2+v}m^{1/2+\alpha+s}n^{1/2+\beta+s}}\frac{ds}{s}\frac{du}{u^{i+1}}\frac{dv}{v^{j+1}}.
\end{eqnarray*}
The sum in the integrand is
\begin{equation*}
-\frac{d}{d\gamma}{\sum_{mlh_1h_2=nk}\!\!\!\!\!\!\!}^{*}\ \ \ \ \frac{\mu(l)\mu(h_2)\mu(k)}{l^{1/2+u}h_{1}^{1/2+v+\gamma}(h_2k)^{1/2+v}m^{1/2+\alpha+s}n^{1/2+\beta+s}}\bigg|_{\gamma=0}.
\end{equation*}
So 
\begin{equation*}
J_{1}^{+}(\alpha,\beta)=-\frac{\varphi^{+}(q)}{L}\sum_{i,j\geq1}\frac{a_ib_ji!j!}{(\log y_1)^i(\log y_2)^j}\frac{\partial}{\partial\gamma}J_1'(\alpha,\beta,\gamma)\bigg|_{\gamma=0},
\end{equation*}
where
\begin{eqnarray*}
J_1'(\alpha,\beta,\gamma)&=&\bigg(\frac{1}{2\pi i}\bigg)^3\int_{(2)}\int_{(2)}\int_{(2)}G(s)g_{\alpha,\beta}^{+}(s)\bigg(\frac{q}{\pi}\bigg)^{s}y_{1}^{u}y_{2}^{v}\\
&&\qquad{\sum_{mlh_1h_2=nk}\!\!\!\!\!\!\!}^{*}\ \ \ \ \frac{\mu(l)\mu(h_2)\mu(k)}{l^{1/2+u}h_{1}^{1/2+v+\gamma}(h_2k)^{1/2+v}m^{1/2+\alpha+s}n^{1/2+\beta+s}}\frac{ds}{s}\frac{du}{u^{i+1}}\frac{dv}{v^{j+1}}.
\end{eqnarray*}
Some standard calculations give
\begin{eqnarray}\label{18}
&&{\sum_{mlh_1h_2=nk}\!\!\!\!\!\!\!}^{*}\ \ \ \ \frac{\mu(l)\mu(h_2)\mu(k)}{l^{1/2+u}h_{1}^{1/2+v+\gamma}(h_2k)^{1/2+v}m^{1/2+\alpha+s}n^{1/2+\beta+s}}=\\
&&\qquad B(\alpha,\beta,\gamma,u,v,s)\frac{\zeta(1+\alpha+\beta+2s)\zeta(1+u+v)\zeta(1+\beta+\gamma+v+s)\zeta(1+2v)}{\zeta(1+\alpha+v+s)\zeta(1+\beta+u+s)\zeta(1+\beta+v+s)\zeta(1+\gamma+2v)},\nonumber
\end{eqnarray}
where $B(\alpha,\beta,\gamma,u,v,s)$ is an arithmetical factor given by some Euler product that is absolutely and uniformly convergent in some product of fixed half-planes containing the origin. We first move the $u$-contour and $v$-contour to $\textrm{Re}(u)=\textrm{Re}(v)=\delta$, and then move the $s$-contour to $\textrm{Re}(s)=-(1-\varepsilon)\delta$, where $\delta,\varepsilon>0$ are some fixed small constants such that the arithmetical factor converges absolutely and $\vartheta_1+\vartheta_2<1-\varepsilon$. In doing so we only cross a simple pole at $s=0$. Note that the simple pole at $s=-(\alpha+\beta)/2$ of $\zeta(1+\alpha+\beta+2s)$ has been cancelled out by the factor $G(s)$. By bounding the integral by absolute values, the contribution along the new line is
\begin{equation*}
\ll q^{-(1-\varepsilon)\delta}(y_1y_{2})^{\delta}\ll q^{(-1+\varepsilon+\vartheta_1+\vartheta_2)\delta}.
\end{equation*}
Thus
\begin{eqnarray*}
J_1'(\alpha,\beta,\gamma)&=&\bigg(\frac{1}{2\pi i}\bigg)^2\int_{(\delta)}\int_{(\delta)}y_{1}^{u}y_{2}^{v}B(\alpha,\beta,\gamma,u,v,0)\\
&&\qquad\frac{\zeta(1+\alpha+\beta)\zeta(1+u+v)\zeta(1+\beta+\gamma+v)\zeta(1+2v)}{\zeta(1+\alpha+v)\zeta(1+\beta+u)\zeta(1+\beta+v)\zeta(1+\gamma+2v)}\frac{du}{u^{i+1}}\frac{dv}{v^{j+1}}+O(q^{-\varepsilon}).
\end{eqnarray*}

We now take the derivative with respect to $\gamma$ and set $\gamma=0$. We first note that by moving the contours to $\textrm{Re}(w_1)=\textrm{Re}(w_2)\asymp L^{-1}$ and bounding the integral with absolute values, we get $J_1'(\alpha,\beta,\gamma)\ll L^{i+j}$. Hence by the Cauchy theorem
\begin{equation}\label{105}
\frac{\partial}{\partial \gamma}J_1'(\alpha,\beta,\gamma)\bigg|_{\gamma=0}=\zeta(1+\alpha+\beta)\bigg(L_{11}(\alpha,\beta)+L_{12}(\alpha,\beta)\bigg)+O(L^{i+j}),
\end{equation}
where
\begin{eqnarray*}
L_{11}(\alpha,\beta)&=&\bigg(\frac{1}{2\pi i}\bigg)^2\int_{(L^{-1})}\int_{(L^{-1})}y_{1}^{u}y_{2}^{v}B(\alpha,\beta,0,u,v,0)\zeta(1+u+v)\\
&&\qquad\qquad\frac{1}{\zeta(1+\beta+u)}\frac{\zeta'(1+\beta+v)}{\zeta(1+\alpha+v)\zeta(1+\beta+v)}\frac{du}{du^{i+1}}\frac{dv}{dv^{j+1}}
\end{eqnarray*}
and
\begin{eqnarray*}
L_{12}(\alpha,\beta)&=&-\bigg(\frac{1}{2\pi i}\bigg)^2\int_{(L^{-1})}\int_{(L^{-1})}y_{1}^{u}y_{2}^{v}B(\alpha,\beta,0,u,v,0)\zeta(1+u+v)\\
&&\qquad\qquad\frac{1}{\zeta(1+\beta+u)}\frac{\zeta'(1+2v)}{\zeta(1+\alpha+v)\zeta(1+2v)}\frac{du}{du^{i+1}}\frac{dv}{dv^{j+1}}.
\end{eqnarray*}

\subsection{Evaluation of $L_{11}(\alpha,\beta)$ and $L_{12}(\alpha,\beta)$}

\begin{lemma}
With $L_{11}(\alpha,\beta)$ and $L_{12}(\alpha,\beta)$ defined above we have
\begin{eqnarray}\label{14}
&&\!\!\!\!\!\!\!\!\!\!\!L_{11}(\alpha,\beta)=-(\log y_1)^{i-1}(\log y_2)^{j+1}\int_{0}^{1}\int_{0}^{x}y_{2}^{-\beta u}\bigg(\frac{\beta(\log y_1) (1-\frac{\vartheta_2(1-x)}{\vartheta_1})^i}{i!}+\frac{(1-\frac{\vartheta_2(1-x)}{\vartheta_1})^{i-1}}{(i-1)!}\bigg)\nonumber\\
&&\qquad\qquad\qquad \bigg(\frac{\alpha(\log y_2)(x-u)^j}{j!}+\frac{(x-u)^{j-1}}{(j-1)!}\bigg)dudx+O(L^{i+j-1})+O(L^{i-1+\varepsilon})
\end{eqnarray}
and
\begin{eqnarray*}
L_{12}(\alpha,\beta)&=&\frac{(\log y_1)^{i-1}(\log y_2)^{j+1}}{2}\int_{0}^{1}\bigg(\frac{\beta(\log y_1) (1-\frac{\vartheta_2(1-x)}{\vartheta_1})^i}{i!}+\frac{(1-\frac{\vartheta_2(1-x)}{\vartheta_1})^{i-1}}{(i-1)!}\bigg)\\
&&\qquad\qquad \bigg(\frac{\alpha(\log y_2)x^{j+1}}{(j+1)!}+\frac{x^j}{j!}\bigg)dx+O(L^{i+j-1})+O(L^{i-1+\varepsilon}).\nonumber
\end{eqnarray*}
\end{lemma}
\begin{proof}
We shall illustrate the proof for $L_{11}(\alpha,\beta)$. The case of $L_{12}(\alpha,\beta)$ can be treated similarly. 

This is more intricate than the previous section. We cannot move the contours into the critical strip due to the pole of $\zeta(1+u+v)$ at $u=-v$. We will need to separate the variables $u$ and $v$ first.

Note that by bounding the integral with absolute values, we get $L_{11}(\alpha,\beta)\ll L^{i+j}$. We denote by $L_{11}'(\alpha,\beta)$ the same integral as $L_{11}(\alpha,\beta)$ but with $B(\alpha,\beta,0,u,v,0)$ being replaced by $B(\alpha,\beta,0,0,0,0)$. Then we have $L_{11}(\alpha,\beta)=L_{11}'(\alpha,\beta)+O(L^{i+j-1})$. We will later check that $B(0,0,0,0,0,0)=1$ (see the end of the section), a result we will use freely from now on.

Changing the orders of summation and integration we get
\begin{equation}\label{16}
L_{11}'(\alpha,\beta)=\sum_{n\leq y_2}\frac{1}{n}N_{11}N_{12},
\end{equation}
where
\begin{equation*}
N_{11}=\frac{1}{2\pi i}\int_{(L^{-1})}\bigg(\frac{y_1}{n}\bigg)^u\frac{1}{\zeta(1+\beta+u)}\frac{du}{u^{i+1}}
\end{equation*}
and
\begin{eqnarray*}
N_{12}=\frac{1}{2\pi i}\int_{(L^{-1})}\bigg(\frac{y_2}{n}\bigg)^v\frac{\zeta'(1+\beta+v)}{\zeta(1+\alpha+v)\zeta(1+\beta+v)}\frac{dv}{v^{j+1}}.
\end{eqnarray*}
We have restricted the summation in \eqref{16} to $n\leq y_2$ by moving the $v$-contour far to the right otherwise.

We first work on $N_{11}$. We note that since $n\leq y_2$, we have $\log y_1/n\geq(\vartheta_1-\vartheta_2)L$. Hence as in Lemma 4.1, the prime number theorem shows that $N_{11}$ is captured by the residue at $u=0$, with an error of size $O(L^{-B})$ for arbitrarily large $B$. Taking the first term in the Taylor series of the zeta-function and using the trivial bound lead to 
\begin{equation*}
N_{11}=\frac{1}{2\pi i}\oint\bigg(\frac{y_1}{n}\bigg)^{u}(\beta+u)\frac{du}{u^{i+1}}+O(L^{i-2}).
\end{equation*}
Thus
\begin{equation*}
N_{11}=\bigg(\frac{\beta(\log y_1/n)^i}{i!}+\frac{(\log y_1/n)^{i-1}}{(i-1)!}\bigg)+O(L^{i-2}).
\end{equation*}

We next evaluate $N_{12}$. This is more subtle since an error term of size $O(1)$ is not sufficient for our purpose. Let $X=o(q)$ be a large parameter which we will specify later. We define the line segments
\begin{eqnarray*}
&&\Gamma_1=\{w=\delta_1+it:|t|\geq X\},\\
&&\Gamma_2=\{w=\sigma\pm iX:\tfrac{-c}{\log X}\leq\sigma\leq\delta_1\}
\end{eqnarray*}
and
\begin{equation*}
\Gamma_3=\{w=\tfrac{-c}{\log X}+it:|t|\leq X\},
\end{equation*}
where $\delta_1\asymp L^{-1}$, and $c$ is some fixed positive constant such that $\zeta(1+\alpha+v)\zeta(1+\beta+v)$ has no zeros, and $1/\zeta(\sigma+it)\ll\log(2+|t|)$ in the region on the right hand side of the contour determined by $\bigcup \Gamma_i$ (see Theorem 3.11 [\textbf{\ref{T}}]). By Cauchy's theorem we have
\begin{equation*}
N_{12}=\textrm{Res}_{v=0}+\textrm{Res}_{v=-\beta}+R_1+R_2+R_3,
\end{equation*}
where $R_1$, $R_2$ and $R_3$ are the contributions along $\Gamma_1$, $\Gamma_2$ and $\Gamma_3$, respectively. We have
\begin{eqnarray*}
R_1\ll (\log X)^2/X^j\ll X^{-j+\varepsilon},\quad R_2\ll (\log X)^2/X^{j+1}\ll X^{-j-1+\varepsilon}
\end{eqnarray*}
and
\begin{equation*}
R_3\ll (\log X)^{j+2}\bigg(\frac{y_2}{n}\bigg)^{-c/\log X}.
\end{equation*}
We choose $X\asymp L$ and obtain an error of size $O(L^{-j+\varepsilon})+O((\tfrac{y_2}{n})^{-\nu}L^\varepsilon)$, for some $\nu\asymp(\log\log q)^{-1}$.

Now the sum of the residues can be written as 
\begin{equation*}
\frac{1}{2\pi i}\oint\bigg(\frac{y_2}{n}\bigg)^{v}\frac{\zeta'(1+\beta+v)}{\zeta(1+\alpha+v)\zeta(1+\beta+v)}\frac{dv}{v^{j+1}},
\end{equation*}
where the contour is a circle with radius $\asymp L^{-1}$ enclosing the origin. This integral is trivially bounded by $O(L^{j})$. Hence by taking the first terms in the Taylor series of the zeta-functions we get
\begin{equation*}
N_{12}=-\frac{1}{2\pi i}\oint\bigg(\frac{y_2}{n}\bigg)^{v}\frac{(\alpha+v)}{(\beta+v)}\frac{dv}{v^{j+1}}+O(L^{j-1})+O\bigg(\bigg(\frac{y_2}{n}\bigg)^{-\nu}L^\varepsilon\bigg).
\end{equation*}
We use the identity
\begin{equation*}
\frac{1}{\beta+v}=\int_{n/y_2}^{1}t^{\beta+v-1}dt+\frac{(y_2/n)^{-\beta-v}}{\beta+v}.
\end{equation*}
The contribution of the second term is
\begin{equation*}
-\bigg(\frac{y_2}{n}\bigg)^{-\beta}\frac{1}{2\pi i}\oint\frac{(\alpha+v)}{(\beta+v)}\frac{dv}{v^{j+1}}.
\end{equation*}
For $j\geq1$, this can be seen to vanish by taking the contour to be arbitrarily large. Thus
\begin{eqnarray*}
N_{12}&=&-\int_{n/y_2}^{1}t^{\beta}\frac{1}{2\pi i}\oint\bigg(\frac{y_2t}{n}\bigg)^{v}(\alpha+v)\frac{dv}{v^{j+1}}\frac{dt}{t}+O(L^{j-1})+O\bigg(\bigg(\frac{y_2}{n}\bigg)^{-\nu}L^\varepsilon\bigg)\\
&=&-\int_{n/y_2}^{1}t^{\beta}\bigg(\frac{\alpha(\log y_2t/n)^j}{j!}+\frac{(\log y_2t/n)^{j-1}}{{(j-1)}!}\bigg)\frac{dt}{t}+O(L^{j-1})+O\bigg(\bigg(\frac{y_2}{n}\bigg)^{-\nu}L^\varepsilon\bigg)
\end{eqnarray*}
Changing the variable $t=(y_2/n)^{-(1-u)}$ we obtain
\begin{eqnarray*}
N_{12}&=&-\bigg(\log \frac{y_2}{n}\bigg)\int_{0}^{1}\bigg(\frac{y_2}{n}\bigg)^{-\beta(1-u)}\bigg(\frac{\alpha(u\log y_2/n)^{j}}{j!}+\frac{(u\log y_2/n)^{j-1}}{(j-1)!}\bigg)du\\
&&\qquad\qquad+O(L^{j-1})+O\bigg(\bigg(\frac{y_2}{n}\bigg)^{-\nu}L^\varepsilon\bigg).
\end{eqnarray*}

Collecting the evaluations of $N_{11}$ and $N_{12}$ we get
\begin{eqnarray*}
N_{11}N_{12}&=&-\bigg(\log \frac{y_2}{n}\bigg)\bigg(\frac{\beta(\log y_1/n)^i}{i!}+\frac{(\log y_1/n)^{i-1}}{(i-1)!}\bigg)\\
&&\qquad\int_{0}^{1}\bigg(\frac{y_2}{n}\bigg)^{-\beta(1-u)}\bigg(\frac{\alpha(u\log y_2/n)^{j}}{j!}+\frac{(u\log y_2/n)^{j-1}}{(j-1)!}\bigg)du\\
&&\qquad\qquad\qquad+O(L^{i+j-2})+O\bigg(\bigg(\frac{y_2}{n}\bigg)^{-\nu}L^{i-1+\varepsilon}\bigg).
\end{eqnarray*}
By Lemma 3.7, the contribution of the $O$-terms to $L_{11}'(\alpha,\beta)$ in \eqref{16} is
\begin{equation*}
\ll O(L^{i+j-1})+O(L^{i-1+\varepsilon}).
\end{equation*}
For the main term, we use Lemma 3.6. Simplifying we obtain \eqref{14}.
\end{proof}

\subsection{Simplification}

We collect the evaluations of $L_{11}(\alpha,\beta)$ and $L_{12}(\alpha,\beta)$ to obtain
\begin{equation*}
J_{1}^{+}(\alpha,\beta)=\frac{\vartheta_2\varphi^{+}(q)\zeta(1+\alpha+\beta)}{\vartheta_1 L}J_{1}''(\alpha,\beta)+O(qL^{-1+\varepsilon}),
\end{equation*}
where $J_{1}''(\alpha,\beta)$ is equal to
\begin{eqnarray*}
&&\!\!\!\!\!\!\!\!\!\!\!\!\int_{0}^{1}\int_{0}^{x}y_{2}^{-\beta u}\bigg(\beta(\log y_1)P\big(1-\tfrac{\vartheta_2(1-x)}{\vartheta_1}\big)+P'\big(1-\tfrac{\vartheta_2(1-x)}{\vartheta_1}\big)\bigg)\bigg(\alpha(\log y_2)Q(x-u)+Q'(x-u)\bigg)dudx\\
&&\!\!\!\!\!\!\!\!\!\!\!\!\qquad-\tfrac{1}{2}\int_{0}^{1}\bigg(\beta(\log y_1)P\big(1-\tfrac{\vartheta_2(1-x)}{\vartheta_1}\big)+P'\big(1-\tfrac{\vartheta_2(1-x)}{\vartheta_1}\big)\bigg)\bigg(\alpha(\log y_2)Q_1(x)+Q(x)\bigg)dx.
\end{eqnarray*}
We write this in a more compact form as
\begin{eqnarray*}
J_{1}^{+}(\alpha,\beta)&=&\frac{\varphi^{+}(q)\zeta(1+\alpha+\beta)}{L}\frac{d^2}{dadb}\bigg\{\tfrac{\vartheta_2}{\vartheta_1}\int_{0}^{1}\int_{0}^{x}y_{1}^{\beta a}y_{2}^{\alpha b-\beta u}P\big(1-\tfrac{\vartheta_2(1-x)}{\vartheta_1}+a\big)Q(x-u+b)dudx\\
&&\qquad-\tfrac{\vartheta_2}{2\vartheta_1}\int_{0}^{1}y_{1}^{\beta a}y_{2}^{\alpha b}P\big(1-\tfrac{\vartheta_2(1-x)}{\vartheta_1}+a\big)Q_1(x+b)dudx\bigg\}\bigg|_{a=b=0}+O(qL^{-1+\varepsilon}).
\end{eqnarray*}

Next we combine $J_{1}^{+}(\alpha,\beta)$ and $J_{1}^{-}(\alpha,\beta)$. We recall that essentially $J_{1}^{-}(\alpha,\beta)=\big(q^{-\alpha-\beta}+O(L^{-1})\big)J_{1}^{+}(-\beta,-\alpha)$. Writing
\begin{equation*}
U(\alpha,\beta;u)=\frac{y_{1}^{\beta a}y_{2}^{\alpha b-\beta u}-q^{-\alpha-\beta}y_{1}^{-\alpha a}y_{2}^{-\beta b+\alpha u}}{\alpha+\beta}.
\end{equation*}
Using the integral formula
\begin{equation}\label{24}
\frac{1-z^{-\alpha-\beta}}{\alpha+\beta}=(\log z)\int_{0}^{1}z^{-(\alpha+\beta)t}dt,
\end{equation}
we have
\begin{eqnarray*}
U(\alpha,\beta;u)=Ly_{1}^{\beta a}y_{2}^{\alpha b-\beta u}\big(1+\vartheta_1 a+\vartheta_2(b-u)\big)\int_{0}^{1}(qy_{1}^{a}y_{2}^{b-u})^{-(\alpha+\beta)t}dt.
\end{eqnarray*}
Simplifying we obtain \eqref{19}.

Finally we need to verify that $B(0,0,0,0,0,0)=1$. Letting $\alpha=\beta=\gamma=0$ and $u=v=s$ in \eqref{18} we have
\begin{equation*}
B(0,0,0,s,s,s)={\sum_{mlh_1h_2=nk}\!\!\!\!\!\!\!}^{*}\ \ \ \ \frac{\mu(l)\mu(h_2)\mu(k)}{(mnlh_1h_2k)^{1/2+s}}=1,
\end{equation*}
for all $s$. This completes the proof of Lemma 2.2.

\section{Evaluating $J_2(\alpha,\beta)$}

\subsection{Reduction to a contour integral}

We recall that $J_2(\alpha,\beta)$ is defined by \eqref{7}. Writing out the definition of $M_2(\chi)$ we have
\begin{eqnarray*}
J_2(\alpha,\beta)&=&\frac{1}{L^2}\sum_{h_1k_1,h_2k_2\leq y_2}\frac{(\log*\mu)(h_1)\mu(k_1)Q[h_1k_1](\log*\mu)(h_2)\mu(k_2)Q[h_2k_2]}{\sqrt{h_1k_1h_2k_2}}\\
&&\qquad\qquad\qquad{\sum_{\chi(\textrm{mod}\ \!q)}{\!\!\!\!\!}}^{+}\ L({\scriptstyle{\frac{1}{2}}}+\alpha,\chi)L(\tfrac{1}{2}+\beta,\overline{\chi})\chi(h_2k_1)\overline{\chi}(h_1k_2).
\end{eqnarray*}
By Lemma 3.5, we can write $J_2(\alpha,\beta)=J_{2}^{+}(\alpha,\beta)+J_{2}^{-}(\alpha,\beta)+O((y_{2}^{2}q)^{1/2+\varepsilon})$, where
\begin{equation*}
J_{2}^{+}(\alpha,\beta)=\frac{\varphi^{+}(q)}{L^2}{\sum_{\substack{h_1k_1,h_2k_2\leq y_2\\mh_2k_1=nh_1k_2}}\!\!\!\!\!\!\!\!\!\!}^{*}\ \ \ \  \frac{(\log*\mu)(h_1)\mu(k_1)Q[h_1k_1](\log*\mu)(h_2)\mu(k_2)Q[h_2k_2]}{\sqrt{h_1k_1h_2k_2}m^{1/2+\alpha}n^{1/2+\beta}}W_{\alpha,\beta}^{+}\bigg(\frac{\pi mn}{q}\bigg)
\end{equation*}
and $J_{2}^{-}(\alpha,\beta)$ is equal to
\begin{equation*}
\frac{\varphi^{+}(q)}{L^2}\bigg(\frac{q}{\pi}\bigg)^{-\alpha-\beta}{\sum_{\substack{h_1k_1,h_2k_2\leq y_2\\mh_1k_2=nh_2k_1}}\!\!\!\!\!\!\!\!\!\!}^{*}\ \ \ \  \frac{(\log*\mu)(h_1)\mu(k_1)Q[h_1k_1](\log*\mu)(h_2)\mu(k_2)Q[h_2k_2]}{\sqrt{h_1k_1h_2k_2}m^{1/2-\alpha}n^{1/2-\beta}}W_{\alpha,\beta}^{-}\bigg(\frac{\pi mn}{q}\bigg).
\end{equation*}
Using \eqref{8} and \eqref{12} we obtain
\begin{eqnarray*}
J_{2}^{+}(\alpha,\beta)&=&\frac{\varphi^{+}(q)}{L^2}\sum_{i,j\geq1}\frac{b_ib_ji!j!}{(\log y_2)^{i+j}}\bigg(\frac{1}{2\pi i}\bigg)^3\int_{(2)}\int_{(2)}\int_{(2)}G(s)g_{\alpha,\beta}^{+}(s)\bigg(\frac{q}{\pi}\bigg)^{s}y_{2}^{u+v}\\
&&\qquad{\sum_{mh_2k_1=nh_1k_2}\!\!\!\!\!\!\!\!\!\!}^{*}\ \ \ \ \ \frac{(\log*\mu)(h_1)\mu(k_1)(\log*\mu)(h_2)\mu(k_2)}{(h_1k_1)^{1/2+u}(h_2k_2)^{1/2+v}m^{1/2+\alpha+s}n^{1/2+\beta+s}}\frac{ds}{s}\frac{du}{u^{i+1}}\frac{dv}{v^{j+1}}.
\end{eqnarray*}
The sum in the integrand is
\begin{equation*}
\frac{d^2}{d\gamma_1d\gamma_2}{\sum_{ml_2h_2k_1=nl_1h_1k_2}\!\!\!\!\!\!\!\!\!\!\!\!\!}^{*}\ \ \ \ \ \ \frac{\mu(h_1)\mu(h_2)\mu(k_1)\mu(k_2)}{l_{1}^{1/2+u+\gamma_1}l_{2}^{1/2+v+\gamma_2}(h_1k_1)^{1/2+u}(h_2k_2)^{1/2+v}m^{1/2+\alpha+s}n^{1/2+\beta+s}}\bigg|_{\gamma_1=\gamma_2=0}.
\end{equation*}
As in the previous sections, up to an arithmetical factor $C(\alpha,\beta,\gamma_1,\gamma_2,u,v,s)$, the sum in the integral is
\begin{eqnarray*}
&&\frac{\zeta(1+\alpha+\beta+2s)\zeta(1+\alpha+\gamma_1+u+s)\zeta(1+\beta+\gamma_2+v+s)\zeta(1+\gamma_1+\gamma_2+u+v)}{\zeta(1+\alpha+u+s)\zeta(1+\alpha+v+s)\zeta(1+\gamma_1+u+v)\zeta(1+\gamma_2+u+v)}\\
&&\qquad\qquad \times\frac{\zeta(1+2u)\zeta(1+2v)\zeta^2(1+u+v)}{\zeta(1+\beta+u+s)\zeta(1+\beta+v+s)\zeta(1+\gamma_1+2u)\zeta(1+\gamma_2+2v)}.
\end{eqnarray*}
Here $C(\alpha,\beta,\gamma_1,\gamma_2,u,v,s)$ is an arithmetical factor given by some Euler product that is absolutely and uniformly convergent in some product of fixed half-planes containing the origin. Again we first move the $u$-contour and $v$-contour to $\textrm{Re}(u)=\textrm{Re}(v)=\delta$, and then move the $s$-contour to $\textrm{Re}(s)=-(1-\varepsilon)\delta$, where $\delta,\varepsilon>0$ are some fixed small constants such that the arithmetical factor converges absolutely and $2\vartheta_2<1-\varepsilon$. In doing so we only cross a simple pole at $s=0$ and the contribution along the new line is $O(q^{-\varepsilon})$. We denote
\begin{eqnarray*}
J_2'(\alpha,\beta,\gamma_1,\gamma_2)=\bigg(\frac{1}{2\pi i}\bigg)^2\int_{(\delta)}\int_{(\delta)}y_{2}^{u+v}C(\alpha,\beta,\gamma_1,\gamma_2,u,v,0)\frac{\zeta(1+\gamma_1+\gamma_2+u+v)\zeta^2(1+u+v)}{\zeta(1+\gamma_1+u+v)\zeta(1+\gamma_2+u+v)}
\end{eqnarray*}
\begin{eqnarray*}
\qquad\frac{\zeta(1+\alpha+\beta)\zeta(1+\alpha+\gamma_1+u)\zeta(1+\beta+\gamma_2+v)\zeta(1+2u)\zeta(1+2v)}{\zeta(1+\alpha+u)\zeta(1+\alpha+v)\zeta(1+\beta+u)\zeta(1+\beta+v)\zeta(1+\gamma_1+2u)\zeta(1+\gamma_2+2v)}\frac{du}{u^{i+1}}\frac{dv}{v^{j+1}},
\end{eqnarray*}
so that
\begin{equation*}
J_{2}^{+}(\alpha,\beta)=\frac{\varphi^{+}(q)}{L^2}\sum_{i,j\geq1}\frac{b_ib_ji!j!}{(\log y_2)^{i+j}}\frac{\partial^2}{\partial\gamma_1\partial\gamma_2}J_2'(\alpha,\beta,\gamma_1,\gamma_2)\bigg|_{\gamma_1=\gamma_2=0}+O(q^{1-\varepsilon}).
\end{equation*}

We now take the derivative with respect to $\gamma_1,\gamma_2$ and set $\gamma_1=\gamma_2=0$. We first note that by moving the contours to $\textrm{Re}(u)=\textrm{Re}(v)\asymp L^{-1}$ and bounding the integral with absolute values, we get $J_2'(\alpha,\beta,\gamma_1,\gamma_2)\ll L^{i+j}$. Hence by the Cauchy theorem
\begin{equation*}
\frac{\partial^2}{\partial \gamma_1\partial \gamma_2}J_2'(\alpha,\beta,\gamma_1,\gamma_2)\bigg|_{\gamma_1=\gamma_2=0}=\zeta(1+\alpha+\beta)\bigg(L_{21}(\alpha,\beta)+L_{22}(\alpha,\beta)\bigg)+O(L^{i+j+1}),
\end{equation*}
where
\begin{eqnarray*}
L_{21}(\alpha,\beta)&=&\bigg(\frac{1}{2\pi i}\bigg)^2\int\int_{(L^{-1})}y_{2}^{u+v}C(\alpha,\beta,0,0,u,v,0)\\
&&\qquad\bigg(\zeta''(1+u+v)-\frac{\zeta'(1+u+v)^2}{\zeta(1+u+v)}\bigg)\frac{1}{\zeta(1+\beta+u)\zeta(1+\alpha+v)}\frac{du}{u^{i+1}}\frac{dv}{v^{j+1}}
\end{eqnarray*}
and
\begin{eqnarray*}
L_{22}(\alpha,\beta)&=&\bigg(\frac{1}{2\pi i}\bigg)^2\int\int_{(L^{-1})}y_{2}^{u+v}C(\alpha,\beta,0,0,u,v,0)\zeta(1+u+v)\\
&&\qquad\bigg(\frac{\zeta'(1+\alpha+u)}{\zeta(1+\alpha+u)\zeta(1+\beta+u)}-\frac{\zeta'(1+2u)}{\zeta(1+\beta+u)\zeta(1+2u)}\bigg)\\
&&\qquad\qquad\quad\bigg(\frac{\zeta'(1+\beta+v)}{\zeta(1+\alpha+v)\zeta(1+\beta+v)}-\frac{\zeta'(1+2v)}{\zeta(1+\alpha+v)\zeta(1+2v)}\bigg)\frac{du}{u^{i+1}}\frac{dv}{v^{j+1}}.
\end{eqnarray*}

\subsection{Evaluation of $L_{21}(\alpha,\beta)$ and $L_{22}(\alpha,\beta)$}

\begin{lemma}
With $L_{21}(\alpha,\beta)$ defined above we have
\begin{eqnarray*}
L_{21}(\alpha,\beta)&=&\frac{(\log y_2)^{i+j+1}}{2}\frac{d^2}{dadb}\int_{0}^{1}y_{2}^{\alpha b+\beta a}(1-x)^2\frac{(x+a)^{i}}{i!}\frac{(x+b)^{j}}{j!}dx\bigg|_{a=b=0}\\
&&\qquad\qquad+O(L^{i+j})+O(L^{i+\varepsilon})+O(L^{j+\varepsilon}).
\end{eqnarray*}
\end{lemma}

\begin{lemma}
With $L_{22}(\alpha,\beta)$ defined above we have
\begin{eqnarray}\label{200}
&&\!\!\!\!\!\!\!\!L_{22}(\alpha,\beta)=(\log y_2)^{i+j+1}\frac{d^2}{dadb}\bigg\{\int_{0}^{1}\int_{0}^{x}\int_{0}^{x}y_{2}^{\alpha b+\beta a-\alpha u-\beta v}\frac{(x-u+a)^i}{i!}\frac{(x-v+b)^j}{j!}dudx\nonumber\\
&&\qquad-\tfrac{1}{2}\int_{0}^{1}\int_{0}^{x}y_2^{\alpha b+\beta a-\alpha u}\frac{(x-u+a)^i}{i!}\frac{(x+b)^{j+1}}{(j+1)!}dudx\\
&&\qquad-\tfrac{1}{2}\int_{0}^{1}\int_{0}^{x}y_2^{\alpha a+\beta b-\beta u}\frac{ (x+b)^{i+1}}{(i+1)!}\frac{(x-u+a)^j}{j!}dudx\nonumber\\
&&\qquad+\tfrac{1}{4}\int_{0}^{1}y_{2}^{\alpha b+\beta a}\frac{(x+a)^{i+1}}{(i+1)!}\frac{ (x+b)^{j+1}}{(j+1)!}dx\bigg\}\bigg|_{a=b=0}+O(L^{i+j})+O(L^{i+\varepsilon})+O(L^{j+\varepsilon}).\nonumber
\end{eqnarray}
\end{lemma}

\begin{proof}
We shall illustrate the proof for $L_{22}(\alpha,\beta)$. The case of $L_{21}(\alpha,\beta)$ can be treated similarly.

We first note that by bounding the integral with absolute values, we get $L_{22}(\alpha,\beta)\ll L^{i+j+1}$. We denote by $L_{22}'(\alpha,\beta)$ the same integral as $L_{22}(\alpha,\beta)$ but with $C(\alpha,\beta,0,0,u,v,0)$ being replaced by $C(\alpha,\beta,0,0,0,0,0)$. Then we have $L_{22}(\alpha,\beta)=L_{22}'(\alpha,\beta)+O(L^{i+j})$. We will later check that $C(0,0,0,0,0,0,0)=1$ (see the end of the section), a result we will use freely from now on.

Changing the orders of summation and integration we have
\begin{equation}\label{201}
L_{22}'(\alpha,\beta)=\sum_{n\leq y_2}\frac{1}{n}N_2(\alpha,\beta,i)N_2(\beta,\alpha,j),
\end{equation}
where
\begin{equation*}
N_2(\alpha,\beta,i)=\frac{1}{2\pi i}\int_{(L^{-1})}\bigg(\frac{y_{2}}{n}\bigg)^{u}\bigg(\frac{\zeta'(1+\alpha+u)}{\zeta(1+\alpha+u)\zeta(1+\beta+u)}-\frac{\zeta'(1+2u)}{\zeta(1+\beta+u)\zeta(1+2u)}\bigg)\frac{du}{u^{i+1}}.
\end{equation*}
This integral has been implicitly evaluated in the proof of Lemma 5.1. Collecting the information from that we get
\begin{eqnarray*}
N_2(\alpha,\beta,i)&=&-\bigg(\log\frac{y_2}{n}\bigg)\int_{0}^{1}\bigg(\frac{y_2}{n}\bigg)^{-\alpha(1-u)}\bigg(\frac{\beta(u\log y_2/n)^i}{i!}+\frac{(u\log y_2/n)^{i-1}}{(i-1)!}\bigg)du\\
&&\qquad+\tfrac{1}{2}\bigg(\frac{\beta(\log y_2/n)^{i+1}}{(i+1)!}+\frac{(\log y_2/n)^{i}}{i!}\bigg)+O(L^{i-1})+O\bigg(\bigg(\frac{y_2}{n}\bigg)^{-\nu}L^{\varepsilon}\bigg).
\end{eqnarray*}
Hence
\begin{equation*}
N_2(\alpha,\beta,i)N_2(\beta,\alpha,j)=N_{21}+N_{22}+N_{23}+N_{24}+O(L^{i+j-1})+O\bigg(\bigg(\frac{y_2}{n}\bigg)^{-\nu}L^{i+\varepsilon}\bigg)+O\bigg(\bigg(\frac{y_2}{n}\bigg)^{-\nu}L^{j+\varepsilon}\bigg),
\end{equation*}
where
\begin{eqnarray*}
N_{21}&=&\bigg(\log\frac{y_2}{n}\bigg)^2\int_{0}^{1}\int_{0}^{1}\bigg(\frac{y_2}{n}\bigg)^{-\alpha(1-u)-\beta(1-v)}\bigg(\frac{\beta(u\log y_2/n)^i}{i!}+\frac{(u\log y_2/n)^{i-1}}{(i-1)!}\bigg)\\
&&\qquad\qquad\qquad\qquad\bigg(\frac{\alpha(v\log y_2/n)^j}{j!}+\frac{(v\log y_2/n)^{j-1}}{(j-1)!}\bigg)dudv\\
N_{22}&=&-\tfrac{1}{2}\bigg(\log\frac{y_2}{n}\bigg)\int_{0}^{1}\bigg(\frac{y_2}{n}\bigg)^{-\alpha(1-u)}\bigg(\frac{\beta(u\log y_2/n)^i}{i!}+\frac{(u\log y_2/n)^{i-1}}{(i-1)!}\bigg)\\
&&\qquad\qquad\qquad\qquad\bigg(\frac{\alpha(\log y_2/n)^{j+1}}{(j+1)!}+\frac{(\log y_2/n)^{j}}{j!}\bigg)du\\
N_{23}&=&-\tfrac{1}{2}\bigg(\log\frac{y_2}{n}\bigg)\int_{0}^{1}\bigg(\frac{y_2}{n}\bigg)^{-\beta(1-v)}\bigg(\frac{\beta(\log y_2/n)^{i+1}}{(i+1)!}+\frac{(\log y_2/n)^{i}}{i!}\bigg)\\
&&\qquad\qquad\qquad\qquad\bigg(\frac{\alpha(v\log y_2/n)^{j}}{j!}+\frac{(v\log y_2/n)^{j-1}}{(j-1)!}\bigg)dv,\\
\end{eqnarray*}
and
\begin{equation*}
N_{24}=\tfrac{1}{4}\bigg(\frac{\beta(\log y_2/n)^{i+1}}{(i+1)!}+\frac{(\log y_2/n)^{i}}{i!}\bigg)\bigg(\frac{\alpha(\log y_2/n)^{j+1}}{(j+1)!}+\frac{(\log y_2/n)^{j}}{j!}\bigg).
\end{equation*}
By Lemma 3.7, the contribution of the $O$-terms to $L_{22}'(\alpha,\beta)$ in \eqref{201} is
\begin{equation*}
\ll O(L^{i+j})+O(L^{i+\varepsilon})+O(L^{j+\varepsilon}).
\end{equation*}
For the main terms, we use Lemma 3.6. Simplifying we obtain \eqref{200}.
\end{proof}

\subsection{Simplification} 

We collect the evaluations of $L_{21}(\alpha,\beta)$ and $L_{22}(\alpha,\beta)$ and have
\begin{equation*}
J_{2}^{+}(\alpha,\beta)=\frac{\vartheta_2\varphi^{+}(q)\zeta(1+\alpha+\beta)}{L}J_{2}''(\alpha,\beta)+O(qL^{-1+\varepsilon}),
\end{equation*}
where
\begin{eqnarray*}
J_{2}''(\alpha,\beta)&=&\frac{d^2}{dadb}\bigg\{\tfrac{1}{2}\int_{0}^{1}y_{2}^{\alpha b+\beta a}(1-x)^2Q(x+a)Q(x+b)dx\\
&&\!\!\!\!\!\!\!\!\!\!\!\!\!\!\!\!\!\!\!\!\!\!\!\!\!+\int_{0}^{1}\int_{0}^{x}\int_{0}^{x}y_{2}^{\alpha b+\beta a-\alpha u-\beta v}Q(x-u+a)Q(x-v+b)dudvdx\\
&&\!\!\!\!\!\!\!\!\!\!\!\!\!\!\!\!\!\!\!\!\!\!\!\!\!-\tfrac{1}{2}\int_{0}^{1}\int_{0}^{x}y_{2}^{\alpha b+\beta a-\alpha u}Q(x-u+a)Q_1(x+b)dudx\\
&&\!\!\!\!\!\!\!\!\!\!\!\!\!\!\!\!\!\!\!\!\!\!\!\!\!-\tfrac{1}{2}\int_{0}^{1}\int_{0}^{x}y_{2}^{\alpha a+\beta b-\beta u}Q(x-u+a)Q_1(x+b)dudx+\tfrac{1}{4}\int_{0}^{1}y_{2}^{\alpha b+\beta a}Q_1(x+a)Q_1(x+b)dx\bigg\}\bigg|_{a=b=0}.
\end{eqnarray*}

Next we combine $J_{2}^{+}(\alpha,\beta)$ and $J_{2}^{-}(\alpha,\beta)$. Essentially we have $J_{2}^{-}(\alpha,\beta)=\big(q^{-\alpha-\beta}+O(L^{-1})\big)J_{2}^{+}(-\beta,-\alpha)$. We proceed as in the previous section. Writing
\begin{equation*}
V(\alpha,\beta;u,v)=\frac{y_{2}^{\alpha b+\beta a-\alpha u-\beta v}-q^{-\alpha-\beta}y_{2}^{-\beta b-\alpha a+\beta u+\alpha v}}{\alpha+\beta}.
\end{equation*}
Using \eqref{24} we get
\begin{equation*}
V(\alpha,\beta;u,v)=Ly_{2}^{\alpha b+\beta a-\alpha u-\beta v}\big(1+\vartheta_2(a+b-u-v)\big)\int_{0}^{1}(qy_{2}^{a+b-u-v})^{-(\alpha+\beta)t}dt.
\end{equation*}
Simplifying we obtain Lemma 2.3.

Finally we need to verify that $C(0,0,0,0,0,0,0)=1$. Taking $\alpha=\beta=\gamma_1=\gamma_2=0$ and $u=v=s$ we have
\begin{equation*}
C(0,0,0,0,s,s,s)={\sum_{ml_2h_2k_1=nl_1h_1k_2}\!\!\!\!\!\!\!\!\!\!\!\!\!}^{*}\ \ \ \ \ \ \frac{\mu(h_1)\mu(h_2)\mu(k_1)\mu(k_2)}{(mnl_1l_2h_1h_2)^{1/2+s}}=1,
\end{equation*}
for all $s$. This completes the proof of Lemma 2.3.

\specialsection*{Acknowledgments}
The author is grateful to Professor Heath-Brown for many helpful and interesting insights into the mollifer method. The question considered in this paper stemmed from a joint project with Professor Milinovich. Thanks also go to him for various stimulating discussions on the topic.

\end{document}